\documentclass[10pt]{my_elsarticle}

\usepackage{a4wide}

\usepackage{amsmath,amsbsy,amssymb,latexsym}
\usepackage{amsthm}
\usepackage{url}

\newtheorem{thm}{Theorem}
\newtheorem{cor}{Corollary}

\theoremstyle{remark}
\newtheorem{rem}{Remark}

\theoremstyle{definition}

\newcommand{\T}{\mathbb{T}}
\newcommand{\R}{\mathbb{R}}
\newcommand{\Z}{\mathbb{Z}}

\newcommand{\rar}{\mbox{$\rightarrow$}}

\begin{document}

\begin{frontmatter}

\title{Optimality conditions for the calculus of variations\\
with higher-order delta derivatives\footnote{Submitted 26/Jul/2009;
Revised 04/Aug/2010; Accepted 09/Aug/2010;
for publication in \emph{Appl. Math. Lett.}}}

\author[lx]{Rui A. C. Ferreira}
\ead{ruiacferreira@ua.pt}

\author[pl]{Agnieszka B. Malinowska}
\ead{abmalinowska@ua.pt}

\author[av]{Delfim F. M. Torres\corref{cor1}}
\ead{delfim@ua.pt}

\cortext[cor1]{Corresponding author}

\address[lx]{Department of Mathematics,
Faculty of Engineering and Natural Sciences,
Lusophone University of Humanities and Technologies,
1749-024 Lisbon, Portugal}

\address[pl]{Department of Mathematics,
Faculty of Computer Science,
Bia{\l}ystok University of Technology,
15-351 Bia\l ystok, Poland}

\address[av]{Department of Mathematics,
Center for Research and Development in Mathematics and Applications,
University of Aveiro,
Campus Universit\'{a}rio de Santiago,
3810-193 Aveiro, Portugal}


\begin{abstract}
We prove the Euler-Lagrange delta-differential equations
for problems of the calculus of variations on arbitrary time scales
with delta-integral functionals depending on higher-order delta derivatives.
\end{abstract}


\end{frontmatter}


\section{Introduction}

In recent years numerous works have been dedicated to the calculus of variations
on time scales and their generalizations ---
see \cite{MyID:140,MyID:175,MyID:180,MyID:170,MyID:141,MyID:183,MyID:187,MyID:171,MyID:174}
and the references therein.
Most of them deal with delta or nabla derivatives of first-order
\cite{Ric:Del,Atici06,Atici09,Zbig:Del,CD:Bohner:2004,FT:Rem,iso:ts,zeidan,Agn:Del,Agn:Del2},
only a few with higher-order derivatives \cite{FT,MT}.
Depending on the type of the functional being considered, different time scale Euler-Lagrange
type equations are obtained. For variational problems of first-order
the Euler-Lagrange equations are valid for an arbitrary
time scale $\mathbb{T}$, while for the problems with higher-order delta (or nabla)
derivatives they are only valid in a certain class of time scales, more precisely,
the ones for which the forward (or backward) jump operator is a polynomial
of degree one \cite{FT,MT}. Here we consider variational problems
involving Hilger derivatives of higher order, and prove a necessary optimality condition
of the Euler-Lagrange type on an arbitrary time scale, \textrm{i.e.},
without imposing any restriction to the jump operators.


\section{Preliminaries}
\label{sec:Prel}

Here we recall some basic results and notation needed in the sequel.
For the theory of time scales we refer the reader to
\cite{Agarwal,livro,Hilger90,Hilger97}.

A time scale $\T$ is an arbitrary nonempty closed subset of the real
numbers $\R$. The functions $\sigma:\T \rar \T$ and $\rho:\T \rar
\T$ are, respectively, the forward and backward jump operators:
$\sigma(t)=\inf{\{s\in\mathbb{T}:s>t\}}$ with
$\inf\emptyset=\sup\mathbb{T}$ (\textrm{i.e.}, $\sigma(M)=M$ if
$\mathbb{T}$ has a maximum $M$);
$\rho(t)=\sup{\{s\in\mathbb{T}:s<t\}}$ with
$\sup\emptyset=\inf\mathbb{T}$ (\textrm{i.e.}, $\rho(m)=m$ if
$\mathbb{T}$ has a minimum $m$). The symbol $\emptyset$ denotes
the empty set. The graininess function on $\T$ is defined
by $\mu(t):=\sigma(t)-t$. For $\T=\R$
one has $\sigma(t)=t=\rho(t)$ and $\mu(t) \equiv 0$ for any $t \in
\R$. For $\T=\Z$ one has $\sigma(t)=t+1$, $\rho(t)=t-1$, and $\mu(t)
\equiv 1$ for every $t \in \Z$. A point $t\in\mathbb{T}$ is called
right-dense, right-scattered, left-dense, or left-scattered, if
$\sigma(t)=t$, $\sigma(t)>t$, $\rho(t)=t$, or $\rho(t)<t$,
respectively.

Let $\mathbb{T}=[a,b]\cap\mathbb{T}_{0}$ with $a<b$ and
$\mathbb{T}_0$ a time scale. We define
$\mathbb{T}^\kappa:=\mathbb{T}\backslash(\rho(b),b]$, and
$\mathbb{T}^{\kappa^0} := \mathbb{T}$,
$\mathbb{T}^{\kappa^n}:=\left(\mathbb{T}^{\kappa^{n-1}}\right)^\kappa$
for $n\in\mathbb{N}$. The following standard notation is used for
$\sigma$ (and $\rho$): $\sigma^0(t) = t$, $\sigma^n(t) = (\sigma
\circ \sigma^{n-1})(t)$, $n \in \mathbb{N}$.

We say that a function $f:\mathbb{T}\rightarrow\mathbb{R}$ is
\emph{delta-differentiable} at $t\in\mathbb{T}^\kappa$
if there is a number $f^{\Delta}(t)$
such that for all $\varepsilon>0$ there exists a neighborhood $U$
of $t$ such that
$$
\left|f(\sigma(t))-f(s)-f^{\Delta}(t)(\sigma(t)-s)\right|
\leq\varepsilon|\sigma(t)-s|,\quad\mbox{ for all $s\in U$}.
$$
We call $f^{\Delta}(t)$ the \emph{delta-derivative} of $f$ at $t$.
We note that if the number $f^\Delta(t)$ exists then it is unique
in $\mathbb{T}^\kappa$ (see \cite{Hilger90,Hilger97}).
In the special cases $\T=\R$ and $\T=\Z$,
$f^\Delta$ reduces to the standard derivative $f'(t)$
and the forward difference $\Delta f(t) = f(t+1)-f(t)$, respectively.
Whenever $f^\Delta$ exists, the following formula holds:
$f^\sigma(t)=f(t)+\mu(t)f^\Delta(t)$,
where we abbreviate $f\circ\sigma$ by $f^\sigma$.
Let $f^{\Delta^{0}} = f$. We define the
$r$th-delta derivative of $f:\mathbb{T}^{\kappa^r}\rightarrow
\R$, $r\in\mathbb{N}$, to be the function $\left(f^{\Delta^{r-1}}\right)^\Delta$,
provided $f^{\Delta^{r-1}}$ is delta differentiable on $\mathbb{T}^{\kappa^r}$.

A function $f:\mathbb{T} \to \mathbb{R}$ is called rd-continuous if
it is continuous at the right-dense points in $\mathbb{T}$ and its
left-sided limits exist at all left-dense points in $\mathbb{T}$. A
function $f:\mathbb{T} \to \mathbb{R}^n$ is rd-continuous if all its
components are rd-continuous. The set of all rd-continuous functions
is denoted by $C_{rd}$. Similarly, $C^r_{rd}$ will denote the set of
functions with delta derivatives up to order $r$ belonging to
$C_{rd}$. A function $f$ is of class $f\in C_{prd}^r$ if
$f^{\Delta^i}$ is continuous for $i = 0,\ldots,r-1$, and
$f^{\Delta^r}$ exists and is rd-continuous for all, except possibly
at finitely many $t \in \mathbb{T}^{\kappa^r}$.

A piecewise rd-continuous function $f:\mathbb{T} \to \mathbb{R}$
possess an antiderivative $F^{\Delta}=f$, and in this case the delta
integral is defined by $\int_{c}^{d}f(t)\Delta t=F(d)-F(c)$ for all
$c,d\in\T$. It satisfies
$$\int_t^{\sigma(t)}f(\tau)\Delta\tau=\mu(t)f(t).$$
If $\mathbb{T}=\mathbb{R}$, then $\int\limits_{a}^{b}  f(t) \Delta
t=\int\limits_{a}^{b}f(t)dt$, where the integral on the right hand
side is the usual Riemann integral; if $\mathbb{T}=\mathbb{Z}$ and $a<b$,
then $\int\limits_{a}^{b} f(t) \Delta t=\sum\limits_{k=a}^{b-1}f(k)$.


\section{Main results}
\label{sec:mainResults}

Consider the following higher-order problem of the calculus of
variations up to order $r$, $r \ge 1$:
\begin{equation}
\label{eq:prb} \mathcal{L}(y(\cdot)) = \int_{a}^{\rho^{r-1}(b)}
 L(t,y(t),y^\Delta(t),\ldots,y^{\Delta^r}(t))\Delta
t\longrightarrow\min,
\end{equation}
subject to boundary conditions
\begin{equation}
\label{problema }
y(a)=y_a^0, \quad y\left(\rho^{r-1}(b)\right)=y_b^0 \, ,
\cdots \, , y^{\Delta^{r-1}}(a)=y_a^{r-1}, \quad
y^{\Delta^{r-1}}\left(\rho^{r-1}(b)\right)=y_b^{r-1},
\end{equation}
where $\mathbb{T}$ is a bounded time scale with $a:=\min\T$
and $b:=\max\T$, $L:[a,\rho^{r}(b)]_{\T}
\times \mathbb{R}^{r+1}\rightarrow\mathbb{R}$ is a given function,
where we use the notation $[c,d]_{\T}:=[c,d]\cap \T$,
and $y_a^i, y_b^i\in\mathbb{R}$, $i=0,\ldots,r-1$.
The results of the paper are trivially generalized for functions
$y : [a,b]_{\T} \rightarrow\mathbb{R}^n$, but for simplicity of
presentation we restrict ourselves to the scalar case $n=1$.

A function $y(\cdot)\in C_{prd}^{r}$ is said to be admissible if it
is satisfies condition \eqref{problema }.
An admissible $y(\cdot)$ is a \emph{weak local minimizer} for
\eqref{eq:prb}--\eqref{problema } if there exists $\delta>0$ such
that $\mathcal{L}(y(\cdot))\leq\mathcal{L}(\bar{y}(\cdot))$ for any
admissible $\bar{y}\in\textrm{C}_{prd}^r$ with
$\|y-\bar{y}\|_{r,\infty}<\delta$, where
$$||y||_{r,\infty} := \sum_{i=0}^{r} \left\|y^{\Delta^i}\right\|_{\infty},$$
$y^{\Delta^0} = y$ and $||y||_{\infty}:= \sup_{t \in
[a,\rho^{r}(b)]_{\T}} |y(t)|$.
For simplicity of notation we introduce the operator $[y]$  defined by
$[y](t) = \left(t,y(t),y^\Delta(t),\ldots,y^{\Delta^r}(t)\right)$.
Then, functional \eqref{eq:prb} can be written as
\begin{equation*}
\mathcal{L}(y(\cdot)) = \int_{a}^{\rho^{r-1}(b)} L[y](t) \Delta t.
\end{equation*}
We assume that $(t,u_1,\ldots,u_{r+1}) \rightarrow
L(t,u_1,\ldots,u_{r+1})$ has continuous partial derivatives
$\frac{\partial L}{\partial u_{i}}$ for all
$t\in[a,\rho^{r}(b)]_{\T} $, $i=1,\ldots,r+1$,
and $t\rightarrow L[y](t)$ and $t\rightarrow\frac{\partial L}{\partial u_{i}}[y](t)$,
$i=1,\ldots,r+1$, are piecewise rd-continuous for all admissible functions $y(\cdot)$.


\subsection{The higher-order Euler-Lagrange equation}
\label{subsec:HO}

We now prove the Euler-Lagrange equation for problem
\eqref{eq:prb}--\eqref{problema }.

\begin{rem}
\label{rem:Pneeds2rp1points}
In order for the problem
to be nontrivial we require the time scale $\mathbb{T}$
to have at least $2r+1$ points. Indeed, if the time scale has only $2r$
points, then it can be written as
$\mathbb{T}=\{a,\sigma(a),\ldots,\sigma^{2r-1}(a)\}$ and
\begin{multline}
\label{snormal}
\int_{a}^{\rho^{r-1}(b)}
L(t,y(t),y^\Delta(t),\ldots,y^{\Delta^r}(t))\Delta t \\
=\int_{a}^{\sigma^{r}(a)}
L(t,y(t),y^\Delta(t),\ldots,y^{\Delta^r}(t))\Delta t
=\sum_{i=0}^{r-1}\int_{\sigma^i(a)}^{\sigma^{i+1}(a)}L(t,
y(t),y^\Delta(t),\ldots,y^{\Delta^r}(t))\Delta t \\
=\sum_{i=0}^{r-1}(\sigma^{i+1}(a)-\sigma^i(a))L(\sigma^i(a),y(\sigma^i(a)),
y^\Delta(\sigma^i(a)),\ldots,y^{\Delta^r}(\sigma^i(a))).
\end{multline}
Having in mind the boundary conditions and the formula
$f^\Delta(t)=\frac{f(\sigma(t))-f(t)}{\mu(t)},$ we can conclude that
the sum in \eqref{snormal} is constant for every admissible function
$y(\cdot)$.
\end{rem}

\begin{thm}
If $y(\cdot)$ is a weak local minimizer for the problem
\eqref{eq:prb}--\eqref{problema }, then $y(\cdot)$ satisfies the
Euler-Lagrange equation
\begin{multline}
\label{eq:EL} \frac{\partial L}{\partial y^{\Delta^r} }[y](t)
- \int_a^{\sigma(t)} \frac{\partial L}{\partial y^{\Delta^{r-1}}}[y](\tau_r) \Delta \tau_r\\
+ \sum_{i=0}^{r-3} (-1)^i \int_a^{\sigma(t)} \int_a^{\sigma(\tau_r)}
\cdots \int_a^{\sigma(\tau_{r-i})} \frac{\partial L}{\partial
y^{\Delta^{r-2-i}}}[y](\tau_{r-1-i})
\Delta\tau_{r-1-i} \cdots \Delta\tau_{r-1}\Delta\tau_{r}\\
(-1)^r \int_a^{\sigma(t)} \left\{ \int_a^{\sigma(\tau_r)} \left[
\cdots \int_a^{\sigma(\tau_2)} \frac{\partial L}{\partial
y}[y](\tau_1) \Delta\tau_1 + c_1 \cdots \right] \Delta\tau_{r-1}
-(-1)^{r-1} c_{r-1}\right\} \Delta\tau_r - c_r = 0
\end{multline}
for some constants $c_1, \ldots, c_r$ and all $t \in [a,\rho^{r}(b)]_{\T}$.
\end{thm}

\begin{proof}
We first introduce some notation: $y_0(t)=y(t)$,
$y_1(t)=y^\Delta(t)$, \ldots, $y_{r-1}(t)=y^{\Delta^{r-1}}(t)$,
$u(t)=y^{\Delta^r}(t)$. Then problem \eqref{eq:prb}--\eqref{problema
} takes the following form:
\begin{equation*}
\begin{gathered}
\mathcal{L}[y(\cdot)]=\int_{a}^{\rho^{r-1}(b)}L(t,y_0(t),y_1(t),
\ldots,y_{r-1}(t),u(t))\Delta t\longrightarrow\min, \\
 \left\{ \begin{array}{l}
y_i^{\Delta}(t)=y^{i+1}(t), \quad i=0,\ldots,r-2,\\
y_{r-1}^{\Delta}(t)=u(t),
\end{array} \right.\\
y^j(a)=y_a^j,\ y^j\left(\rho^{r-1}(b)\right)=y_b^j,\
j=0,\ldots,r-1 \, .
\end{gathered}
\end{equation*}
With the notation $x=(y_0,y_1,\ldots,y_{r-1})$, our problem
\eqref{eq:prb}--\eqref{problema } can be written as the optimal
control problem
\begin{equation}\label{op:1}
\begin{gathered}
\mathcal{L}[x(\cdot)]=\int_{a}^{\rho^{r-1}(b)}L(t,x(t),u(t))\Delta t\longrightarrow\min, \\
x^\Delta(t)=Ax(t)+Bu(t) \, ,\\
\varphi (x(a),x(\rho^{r-1}(b))= \left[ \begin{array}{l}
x(a)-x_a\\
x(\rho^{r-1}(b))-x_b\end{array} \right]=0 \, ,
\end{gathered}
\end{equation}
where
\begin{equation*}
A = \left(\begin{array}{ccccc}
  0 & 1 & 0 & \cdots & 0 \\
  0 & 0 & 1 & \cdots & 0 \\
  \vdots & \vdots & \vdots & \ddots &  \vdots \\
  0 & 0 & 0 & \cdots & 1 \\
  0 & 0 & 0 & \cdots & 0 \\
\end{array}\right) \, ,
\quad B =\left( \begin{array}{c}
0 \\
\vdots \\
1
\end{array}\right).
\end{equation*}
Note that assumption (A1) of \cite[Theorem~9.4]{zeidan2} holds:
matrix $I+\mu(t)A$ is invertible, and the matrix
$\nabla \varphi(x(a),x(\rho^{r-1}(b))$ has full rank. Therefore, if
$(x(\cdot),u(\cdot))$ is a weak local minimum for \eqref{op:1}, then
there exists a constant $\lambda$ and a function
$p:[a,\rho^{r-1}(b)]_{\T}\rightarrow \R^r$, $p\in C_{prd}^1$,
such that $(\lambda, p(\cdot))\neq 0$ and the following
conditions hold:
\begin{equation*}
-p^{\Delta}(t)=A^Tp^{\sigma}(t)+\lambda\left[\frac{\partial
L}{\partial x}(t,x(t),u(t))\right]^T,
\end{equation*}
\begin{equation}\label{op:3}
B^Tp^{\sigma}(t)+\lambda\frac{\partial L}{\partial u}(t,x(t),u(t))=0
\end{equation}
for all $t\in[a,\rho^{r}(b)]_{\T}$. Consequently, if $y(\cdot)$ is a
weak local minimizer for \eqref{eq:prb}--\eqref{problema }, then
\begin{equation}\label{op:4}
p_{r-1}^{\sigma}(t)=-\lambda\frac{\partial L}{\partial u}[y](t)
\end{equation}
holds for all $t\in [a,\rho^{r}(b)]_{\T}$, where
$p_{r-1}^{\sigma}(t)$ is defined recursively by
\begin{align}
p_{0}^{\sigma}(t)&
=-\int_a^{\sigma(t)}\lambda\frac{\partial L}{\partial y_0}[y](\tau_1)\Delta\tau_1-c_1
\, ,\label{op:5}\\
p_{i}^{\sigma}(t)&=-\int_a^{\sigma(t)}\left[\lambda\frac{\partial
L}{\partial y_i}[y](\tau_{i+1})
+p_{i-1}^{\sigma}(\tau_{i+1})\right]\Delta\tau_{i+1}-c_{i-1}\label{op:6},\
i=1,\ldots,r-1 \, ,
\end{align}
with $c_i$, $i = 0,\ldots, r- 1$, constants. From
\eqref{op:4}--\eqref{op:6} we obtain that equation
\begin{multline}
\label{eq:EL:0} \lambda\frac{\partial L}{\partial u}[y](t)
- \int_a^{\sigma(t)} \lambda\frac{\partial L}{\partial y_{r-1}}[y](\tau_r) \Delta \tau_r\\
+ \sum_{i=0}^{r-3} (-1)^i \int_a^{\sigma(t)} \int_a^{\sigma(\tau_r)}
\cdots \int_a^{\sigma(\tau_{r-i})}\lambda \frac{\partial L}{\partial
y_{r-2-i}}[y](\tau_{r-1-i})
\Delta\tau_{r-1-i} \cdots \Delta\tau_{r-1}\Delta\tau_{r}\\
(-1)^r \int_a^{\sigma(t)} \left\{ \int_a^{\sigma(\tau_r)} \left[
\cdots \int_a^{\sigma(\tau_2)} \lambda\frac{\partial L}{\partial
y_0}[y](\tau_1) \Delta\tau_1 + c_1 \cdots \right] \Delta\tau_{r-1}
-(-1)^{r-1} c_{r-1}\right\} \Delta\tau_r - c_r = 0
\end{multline}
holds for all $t \in [a,\rho^{r}(b)]_{\T}$.
We show next that $\lambda\neq 0$. First observe that if $f\in
C_{prd}^1$ and $f^{\sigma}(t)=0$ for all $t\in [a,b]_{\T}^{\kappa}$,
then $f(t)=0$ for all $t\in [\sigma(a),b]_{\T}$. Suppose,
contrary to our claim, that $\lambda=0$ in equation \eqref{op:3} and
\eqref{op:4}. Then, we can write the system of equations
\begin{equation}
\label{eq:syst:ab} \left\{ \begin{array}{ll}
p_0^\Delta (t)&=0 \, , \\
p_i^\Delta (t)&=-p_{i-1}^\sigma (t), \quad i=1,\ldots, r-1 \, , \\
p_{r-1}^\sigma (t)&=0,
\end{array} \right.
\end{equation}
for all $t\in [a,\rho^{r}(b)]_{\T}$. From the last equation we have
$p_{r-1}(t)=0$, $\forall t\in[\sigma (a),\rho^{r-1}(b)]_{\T}$. This
implies that $p_{r-1}^\Delta(t)=0$, $\forall t\in[\sigma
(a),\rho^{r}(b)]_{\T}$, and consequently $p_{r-2}^\sigma(t)=0$,
$\forall t\in[\sigma (a),\rho^{r}(b)]_{\T}$. Therefore,
$p_{r-2}(t)=0$, $\forall t\in[\sigma^2 (a),\rho^{r-1}(b)]_{\T}$.
Repeating this procedure we have $p_{1}(t)=0$ for all
$t\in[\sigma^{r-1}(a),\rho^{r-1}(b)]_{\T}$. Hence,
$0=p_{1}^\Delta(t)=-p_0^\sigma (t)=-p_0^\Delta
(t)\mu(t)-p_0(t)=-p_0(t)$ for all
$t\in[\sigma^{r-1}(a),\rho^{r}(b)]_{\T}$. Note that the first
equation of \eqref{eq:syst:ab} implies $p_0(t)=c$ for some constant
$c$ and all $t\in[a,\rho^{r-1}(b)]_{\T}$. Since the time scale has
at least $2r+1$ points (see Remark~\ref{rem:Pneeds2rp1points}), the
set $t\in[\sigma^{r-1}(a),\rho^{r-1}(b)]_{\T}$ is nonempty and we
conclude that $p_0(t)=0$ for all $t\in[a,\rho^{r-1}(b)]_{\T}$.
Substituting this into the second equation we get $p_1^\Delta (t)=d$ for
some constant $d$ and all $t\in[a,\rho^{r-1}(b)]_{\T}$. Having in
mind that $p_1(t_0)=0$ for some $t_0 \in[a,\rho^{r-1}(b)]_{\T}$ we
obtain $p_1(t)=0$ for all $t \in[a,\rho^{r-1}(b)]_{\T}$. Repeating
this procedure we conclude that $p_i(t)=0$, $i=1,\ldots,r-1$, for
all $t\in[a,\rho^{r-1}(b)]_{\T}$. This contradicts the fact that
$(\lambda,p(\cdot))\neq 0$. Hence, equation \eqref{eq:EL:0} can be
divided by $\lambda$ and \eqref{eq:EL} is proved.
\end{proof}


\subsection{Corollaries}
\label{subsec:appl}

For illustrating purposes we consider now the two simplest
situations, \textrm{i.e.}, $r=1$ and $r = 2$.

\begin{cor}[\textrm{cf.} \cite{CD:Bohner:2004,zeidan}]
If $y(\cdot)$ is a weak local minimizer for the problem
\begin{equation*}
\mathcal{L}(y(\cdot)) = \int_{a}^{b}
 L(t,y(t),y^\Delta(t))\Delta
t\longrightarrow\min
\end{equation*}
subject to boundary conditions $y(a)=y_a$ and $y(b)=y_b$,
then $y(\cdot)$ satisfies the Euler-Lagrange equation
\begin{equation*}
\frac{\partial L}{\partial y^\Delta}\left(t,y(t),y^\Delta(t)\right)
= \int_a^{\sigma(t)} \frac{\partial L}{\partial y}\left(\tau,y(\tau),
y^\Delta(\tau)\right) \Delta \tau + c_1
\end{equation*}
for some constant $c_1$ and all $t \in [a,b]_{\T}^\kappa$.
\end{cor}

\begin{cor}[\textrm{cf.} \cite{FT,MT}]
If $y(\cdot)$ is a weak local minimizer for the problem
\begin{equation*}
\mathcal{L}(y(\cdot)) = \int_{a}^{\rho(b)}
 L(t,y(t),y^\Delta(t),y^{\Delta \Delta})\Delta
t\longrightarrow\min
\end{equation*}
subject to boundary conditions
$y(a)=y_a^0$, $y(\rho(b))=y_b$,
$y^\Delta(a)=y_a^1$, and $y^\Delta(\rho(b))=y_b^1$,
then $y(\cdot)$ satisfies the Euler-Lagrange equation
\begin{multline*}
\frac{\partial L}{\partial
y^{\Delta\Delta}}\left(t,y(t),y^\Delta(t),y^{\Delta\Delta}(t)\right)
- \int_a^{\sigma(t)} \frac{\partial L}{\partial
y^\Delta}\left(\tau_2,y(\tau_2),
y^\Delta(\tau_2),y^{\Delta\Delta}(\tau_2)\right)
\Delta \tau_2 \\
+ \int_a^{\sigma(t)} \left[ \int_a^{\sigma(\tau_2)} \frac{\partial
L}{\partial y}\left(\tau_1,y(\tau_1),
y^\Delta(\tau_1),y^{\Delta\Delta}(\tau_1)\right)\Delta\tau_1 + c_1
\right] \Delta\tau_2 - c_2 = 0
\end{multline*}
for some constants $c_1$ and $c_2$ and all $t \in [a,\rho(b)]_{\T}^\kappa$.
\end{cor}


\subsection{An example}
\label{subsec:ex}

Let $\T=[a,b]\cap h\Z$, where $h\mathbb{Z}:=\{h z | z \in
\mathbb{Z}\}$, $h>0$. Then for any $f\in C_{prd}^{r}$ we
have
\begin{align}
\label{eq:exGC} {\underbrace{\left[\int_a^{\sigma(t)}
\left(\int_a^\sigma \cdots \int_a^\sigma f \right) \Delta
\tau\right]}_{j-i\text{ integrals}}}^{\Delta^j} = f^{\Delta^i
\sigma^{j-i}} \, , \quad i \in \{0,\ldots,j-1\} \, ,
\end{align}
where $f^{\Delta^i \sigma^{j-i}}(t)$ stands for
$f^{\Delta^i}(\sigma^{j-i}(t))$. We will show this by induction. For
$j = 1$
\begin{equation*}
\int_a^{\sigma(t)}f(\xi)\Delta\xi=\int_a^{t}f(\xi)\Delta\xi
+\int_t^{t+h}f(\xi)\Delta\xi=\int_a^{t}f(\xi)\Delta\xi+hf(t),
\end{equation*}
and then $\left[\int_a^{\sigma(t)}f(\xi)\Delta\xi\right]^\Delta
=f(t)+hf^\Delta(t) = f^\sigma$. Now assume that \eqref{eq:exGC} is
true for all $j = 1,\ldots,k$. Then for $j=k+1$
\begin{multline*}
{\underbrace{\left[\int_a^{\sigma(t)} \left( \int_a^\sigma \cdots
\int_a^\sigma f \right)
\Delta \tau\right]}_{k+1-i\text{ integrals}}}^{\Delta^{k+1}}
= \left( \underbrace{\int_a^{t} \int_a^\sigma \cdots
\int_a^\sigma}_{k+1-i} f \Delta \tau +
h\underbrace{\int_a^{\sigma(t)} \cdots \int_a^\sigma}_{k-i} f \Delta
\tau
\right)^{\Delta^{k+1}} \\
= \left(\underbrace{\int_a^{\sigma(t)} \cdots \int_a^\sigma}_{k-i}
f \Delta \tau \right)^{\Delta^{k}} +
\left[h\left(\underbrace{\int_a^{\sigma(t)} \cdots
\int_a^\sigma}_{k-i} f \Delta \tau
\right)^{\Delta^{k}}\right]^{\Delta}
= f^{\Delta^i \sigma^{k-i}} + \left(hf^{\Delta^i \sigma^{k-i}}\right)^\Delta
= f^{\Delta^i \sigma^{k+1-i}} \, .
\end{multline*}
Delta differentiating $r$ times both sides of equation \eqref{eq:EL}
and in view of \eqref{eq:exGC}, we obtain the $h$-Euler-Lagrange
equation in delta differentiated form:
\begin{equation*}
L_{y^{\Delta^r}}^{\Delta^r}(t,y,y^\Delta,\ldots,y^{\Delta^r}) +
\sum_{i=0}^{r-1} (-1)^{r-i} L_{y^{\Delta^{i}}}^{\Delta^i
\sigma^{r-i}}(t,y,y^\Delta,\ldots,y^{\Delta^r}) =0.
\end{equation*}


\section*{Acknowledgments}

This work was partially supported by the \emph{Portuguese Foundation
for Science and Technology} (FCT) through the \emph{Center for Research
and Development in Mathematics and Applications} (CIDMA) of University of Aveiro.
The first author was also supported by FCT through the PhD fellowship
SFRH/BD/39816/2007; the second author is currently a researcher
at the University of Aveiro with the support
of Bia{\l}ystok University of Technology, via a project of
the Polish Ministry of Science and Higher Education ``Wsparcie
miedzynarodowej mobilnosci naukowcow''; the third author
was partially supported by the Portugal--Austin (USA)
project UTAustin/MAT/0057/2008.



\end{document}